\documentclass[12pt]{amsart}

\usepackage{amsmath,mathtools,amsthm,amsfonts,amssymb,verbatim,cases,graphicx,multirow}
\usepackage[usenames,dvipsnames]{xcolor}

\usepackage[letterpaper]{geometry}
\usepackage{url}

\newtheorem{theorem}[equation]{Theorem}
\newtheorem{corollary}[equation]{Corollary}

\newtheorem{lemma}[equation]{Lemma}
\newtheorem{proposition}[equation]{Proposition}

\theoremstyle{definition}

\newtheorem{exam}[equation]{Example}
\numberwithin{equation}{section}

\definecolor{mjo}{rgb}{0,0,.9}
\newcommand{\ZZ}{\mathbb{Z}}

\newcommand{\SYM}{\operatorname{Sym}}
\newcommand{\ALT}{\operatorname{Alt}}

\global\long\def\dng{\text{\sf DNG}}
\global\long\def\mex{\operatorname{mex}}
\global\long\def\nim{\operatorname{nim}}
\global\long\def\opt{\operatorname{Opt}}

\global\long\def\pty{\operatorname{pty}}
\global\long\def\type{\operatorname{type}}

\global\long\def\spr{d}

\global\long\def\dih{\operatorname{Dih}}

\newcommand{\notes}[1]{}


\begin{document}

\setlength{\jot}{0pt} 

\title
[Impartial avoidance games for generating finite groups]
{Impartial avoidance games for generating finite groups}

\author{Bret J.~Benesh}
\author{Dana C.~Ernst}
\author{N\'andor Sieben}

\address{
Department of Mathematics,
College of Saint Benedict and Saint John's University,
37 College Avenue South,
Saint Joseph, MN 56374-5011, USA
}
\email{bbenesh@csbsju.edu}
\address{
Department of Mathematics and Statistics,
Northern Arizona University PO Box 5717,
Flagstaff, AZ 86011-5717, USA
}
\email{Dana.Ernst@nau.edu,Nandor.Sieben@nau.edu}

\thanks{This work was conducted during the third author's visit to DIMACS partially enabled through support from the National Science Foundation under grant number \#CCF-1445755.}

\date{\today}


\begin{abstract}
We study an impartial avoidance game introduced by Anderson and Harary.
The game is played by two players who alternately select previously unselected
elements of a finite group.
The first player who cannot select an element without making the set of jointly-selected
elements into a generating set for the group loses the game.
We develop criteria on the maximal subgroups that determine the nim-numbers of these games and use our criteria to study our game for several families of groups, including nilpotent, sporadic, and symmetric groups. 
\end{abstract}


\maketitle


\begin{section}{Introduction}


Anderson and Harary~\cite{anderson.harary:achievement} introduced a two-player
impartial game called \emph{Do Not Generate}. In this avoidance game, two players alternately take 
turns selecting previously unselected elements of a finite group until the group is generated
by the jointly-selected elements. The goal of the game is to avoid generating the group, and the 
first player who cannot select an element without building a generating set loses. 

In the original description of the avoidance game, the game ends when a generating set is built. 
This suggests mis\`ere-play convention. We want to find the nim-values of these games under normal-play, 
so our version does not allow the creation of a generating set---the game
ends when there are no available moves. The two versions have the same outcome, so the difference  
is only a question of viewpoint.

The outcomes of avoidance games were determined for finite abelian groups in~\cite{anderson.harary:achievement}, while Barnes~\cite{Barnes} provided
a criterion for determining the outcome for an arbitrary finite group. He also applied his criterion to determine the outcomes
for some of the more familiar finite groups, including abelian, dihedral, symmetric, and alternating groups, although his analysis is incomplete
for alternating groups.

The fundamental problem in the theory of impartial combinatorial games is finding the nim-number of the game. 
The nim-number determines the outcome of the game and also allows for the easy calculation of the nim-numbers of game sums.
The last two authors~\cite{ErnstSieben} developed tools for studying the nim-numbers of avoidance games, which they applied to certain groups including abelian and dihedral groups.  

Our aim is to find Barnes-like criteria for determining not only the outcomes but also the nim-numbers of avoidance games.
We start by recalling results from~\cite{ErnstSieben} on cyclic groups and groups of odd order, and then we reformulate Barnes' condition in terms of maximal subgroups (Proposition~\ref{prop:DNGMaximalCover}) in the case of non-cyclic groups.  This allows us to determine the nim-numbers of the avoidance games for non-cyclic groups that satisfy Barnes' condition (Proposition \ref{prop:NotCycNotCov}).
Then we combine this with the results of~\cite{ErnstSieben} for cyclic groups to get a complete classification (Theorem~\ref{thm:DNGClassification}) 
of the possible values of the avoidance game. Next, we reformulate our classification in terms of maximal subgroups (Corollary~\ref{cor:DNGCriteria}), and
we also provide a practical checklist for determining the nim-number corresponding to a given group (Proposition~\ref{cor:checklist}). We then apply our theoretical results for several families of groups in Section~\ref{section:Applications}, which is followed by a section whose main result is about quotient groups (Proposition~\ref{prop:FrattiniQuotients}).  We end with several open questions in Section~\ref{sec:questions}.   

Our development is guided by the following intuitive understanding of the game.
At the end of an avoidance game, the players will realize that they simply took turns selecting elements from a single maximal subgroup. 
If they knew in advance which maximal subgroup was going to remain at the end of the game, then the game would be no more complicated than a simple subtraction 
game~\cite[Section 7.6]{albert2007lessons} on a single pile where the players can only remove one object on each turn. The first player wins if that maximal 
subgroup has odd order and the second player wins otherwise. Therefore, the game is a really a struggle to determine which maximal subgroup will 
remain at the end. Viewed this way, the second player's best strategy is to select an element of even order, if possible. It is intuitively clear 
that the first player has a winning strategy if all of the maximal subgroups have odd order, and the second player has a winning strategy if all 
of the maximal subgroups have even order. 

The authors thank John Bray, Derek Holt, and the anonymous referee for suggestions that greatly improved the paper.

\end{section}


\begin{section}{Preliminaries}


\subsection{Impartial games}


We briefly recall the basic terminology of impartial games. 
A comprehensive treatment of impartial games can be
found in~\cite{albert2007lessons,SiegelBook}. An \emph{impartial
game} is a finite set $X$ of \emph{positions} together with a starting
position and a collection $\{\opt(P)\subseteq X\mid P\in X\}$, where $\opt(P)$ is the set of
possible \emph{options} for a position $P$. Two players take turns replacing the current position $P$ with one of the available options in $\opt(P)$. The player who encounters
an empty option set cannot move and therefore \emph{loses}. All games
must come to an end in finitely many turns, so we do not allow infinite
lines of play. 
An impartial game is an \emph{N-position} if the next player wins and it 
is a \emph{P-position} if the previous player wins.

The \emph{minimum excludant} $\mex(A)$ of a set $A$ of ordinals
is the smallest ordinal not contained in the set. The\emph{ nim-number}
$\nim(P)$ of a position $P$ is the minimum excludant of the set
of nim-numbers of the options of $P$. That is, 
\[
\nim(P):=\mex\{\nim(Q)\mid Q\in\opt(P)\}.
\]
Note that the minimum excludant of the empty set is $0$, so the terminal positions
of a game have nim-number $0$. The \emph{nim-number of a game} is the
nim-number of its starting position. The nim-number of a game determines
the outcome of a game since a position $P$ is a P-position if and
only if $\nim(P)=0$; an immediate consequence of this is that the second player has a winning strategy if and only if the nim-number for a game is $0$. 

The \emph{sum} of the games $P$ and $R$ is the game $P+R$ whose
set of options is 
\[
\opt(P+R):=\{Q+R\mid Q\in\opt(P)\}\cup\{P+S\mid S\in\opt(R)\}.
\]
We write $P=R$ if $P+R$ is a P-position. 

The one-pile NIM game with $n$ stones is denoted
by the \emph{nimber} $*n$. The set of options of $*n$ is $\opt(*n)=\{*0,\ldots,*(n-1)\}$.
The fundamental Sprague--Grundy Theorem~\cite{albert2007lessons,SiegelBook} states that $P=*\nim(P)$ for every impartial game $P$.


\subsection{Avoidance games for groups}


We now give a more precise description of the avoidance game $\dng(G)$ played on a nontrivial finite group $G$. We also recall some definitions and results from~\cite{ErnstSieben}. The positions of $\dng(G)$ are exactly the non-generating subsets of $G$; these are the sets of jointly-selected elements.  
The starting position is the empty set since neither player has chosen an element yet.
The first player chooses $x_{1}\in G$ such that $\langle x_{1}\rangle\neq G$ and, at the $k$th turn, the designated player selects $x_{k}\in G\setminus\{x_{1},\ldots,x_{k-1}\}$, 
such that $\langle x_{1},\ldots,x_{k}\rangle\neq G$. A position $Q$ is an option of $P$ if $Q=P \cup \{g\}$ for some $g \in G \setminus P$.  The player who cannot select an element without building a generating set loses the game. We note that there is no avoidance game for the trivial group since the empty set generates the whole group.

Because the game ends once the set of selected elements becomes a maximal subgroup, the set $\mathcal{M}$ of maximal subgroups play a significant role in the game.
The last two authors~\cite{ErnstSieben} define the set
\[
\mathcal{I}:=\{\cap\mathcal{N}\mid\emptyset\not=\mathcal{N\subseteq\mathcal{M}}\}
\]
of \emph{intersection subgroups}, which is the set of all possible intersections of maximal subgroups.
The smallest intersection subgroup is the Frattini subgroup $\Phi(G)$ of $G$, which is the intersection of all maximal subgroups of $G$. 

\begin{figure}
\includegraphics[scale=0.9]{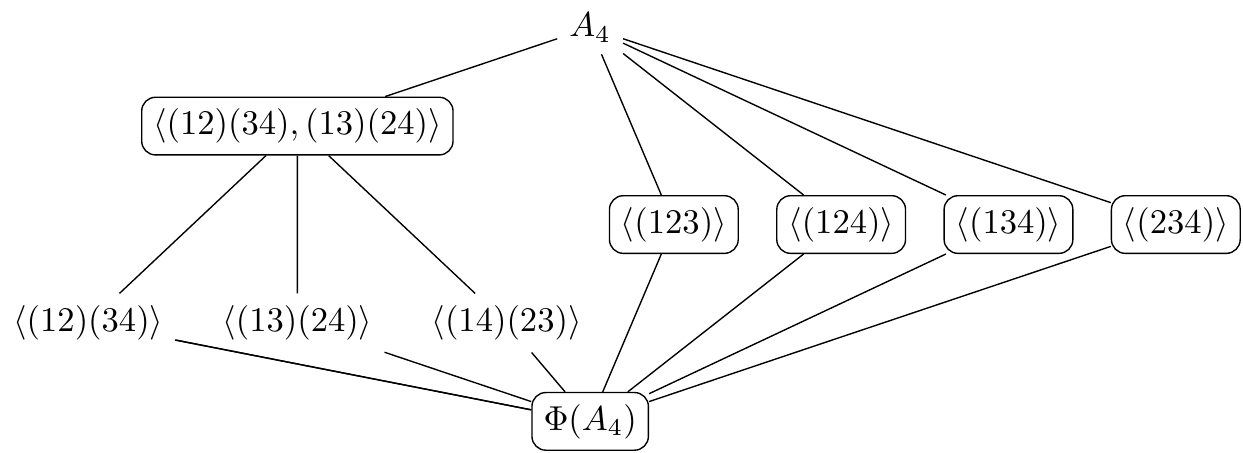}
\caption{\label{A4lattice} Subgroup lattice of $A_4$ with the intersection subgroups circled.}
\end{figure}

\begin{exam}
Subgroups---even those that contain $\Phi(G)$---need not be intersection subgroups.  
The maximal subgroups of $A_4$ comprise four subgroups of order $3$ and one subgroup of order $4$, as shown in Figure~\ref{A4lattice}. 
The subgroup $H=\langle (1,2)(3,4) \rangle$ contains the trivial Frattini subgroup $\Phi(A_4)$, 
but no intersection of some subset of the five maximal subgroups yields a subgroup of order $|H|=2$. So, $H$ is not an intersection subgroup of $A_4$. 
\end{exam}

The set $\mathcal{I}$ of intersection subgroups is partially
ordered by inclusion. We use interval notation to denote certain subsets
of $\mathcal{I}$. For example, if $I\in\mathcal{I}$, then $(-\infty,I):=\{J\in\mathcal{I}\mid J\subset I\}$. 

For each $I\in\mathcal{I}$ let
\[
X_{I}:=\mathcal{P}(I)\setminus\cup\{\mathcal{P}(J)\mid J\in(-\infty,I)\}
\]
be the collection of those subsets of $I$ that are not contained
in any other intersection subgroup smaller than $I$. We let $\mathcal{X}:=\{X_{I}\mid I\in\mathcal{I}\}$
and call an element of $\mathcal{X}$ a \emph{structure class}. 

Parity plays a crucial role in the theory of impartial games.
We define the \emph{parity of a natural number $n$} via $\pty(n):=(1+(-1)^{n+1})/2$. The \emph{parity of a set}
is the parity of the size of the set. 
We will say that the set is \emph{even} or \emph{odd} according to the parity of the set.
Observe that an option of a position has the opposite parity.
The \emph{parity of a structure class} is defined to be $\pty(X_{I}):=\pty(I)$. 

The set $\mathcal{X}$ of structure classes is a partition of the
set of game positions of $\dng(G)$. The starting position $\emptyset$
is in $X_{\Phi(G)}$. The partition $\mathcal{X}$ is compatible with
the option relationship between game positions~\cite[Corollary~3.11]{ErnstSieben}: if $X_{I},X_{J}\in\mathcal{X}$
and $P,Q\in X_{I}\ne X_{J}$, then $\opt(P)\cap X_{J}\not=\emptyset$
if and only if $\opt(Q)\cap X_{J}\ne\emptyset$. 

We say that $X_{J}$ is an \emph{option} of $X_{I}$ and we write
$X_{J}\in\opt(X_{I})$ if $\opt(I)\cap X_{J}\not=\emptyset$. The
\emph{structure digraph} of $\dng(G)$ has vertex set $\{X_{I}\mid I\in\mathcal{I}\}$
and edge set $\{(X_{I},X_{J})\mid X_{J}\in\opt(X_{I})\}$.

If $P,Q\in X_{I}\in\mathcal{X}$ and $\pty(P)=\pty(Q)$, then $\nim(P)=\nim(Q)$ by~\cite[Proposition~3.15]{ErnstSieben}.
In a \emph{structure diagram,} a structure class $X_{I}$ is represented
by a triangle pointing down if $I$ is odd and by a triangle pointing
up if $I$ is even. The triangles are divided into a smaller triangle
and a trapezoid, where the smaller triangle represents the odd positions
of $X_{I}$ and the trapezoid represents the even positions of $X_{I}$.
The numbers in the smaller triangle and the trapezoid are the nim-numbers
of these positions. There is a directed edge from $X_{I}$ to $X_{J}$
provided $X_{J}\in\opt(X_{I})$.  See Figure~\ref{fig:C18xC2}(c) for an example of a structure diagram. 

The \emph{type} of the structure class $X_{I}$ is the triple 
\[
\type(X_{I}):=(\pty(I),\nim(P),\nim(Q)),
\]
where $P,Q\in X_{I}$ with $\pty(P)=0$ and $\pty(Q)=1$. Note that the type of a structure class $X_I$ is determined by the parity of $X_I$ and the types of the options of $X_I$ as shown in Figure~\ref{fig:type}.

\begin{figure}
\includegraphics{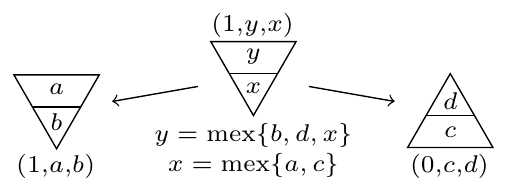}
\caption{\label{fig:type}Example for the calculation of the type of a structure diagram using the types of the options.
}
\end{figure}

The nim-number of the game is the same as the nim-number of the initial position $\emptyset$, which is an even subset of $\Phi(G)$. Because of this, the nim-number of the game is the second component of $\type(X_{\Phi(G)})$, which corresponds to the trapezoidal part of the triangle representing the source vertex $X_{\Phi(G)}$ of the structure diagram.

Loosely speaking, the \emph{simplified structure diagram} of $\dng(G)$ is built from the structure diagram by identifying two structure 
classes that have the same type and the same collection consisting of option types together with the type of the structure class itself. We remove any resulting loops to obtain a simple 
graph as described in~\cite{ErnstSieben}. See Figure~\ref{fig:C18xC2}(d) for an example of a simplified structure diagram.

\end{section}


\begin{section}{Groups of odd order}


This next result is the foundation for our brief study of groups of odd order.

\begin{proposition}\label{prop:TypeOnlyInOddMaximals}
If $X_I$ is a structure class of $\dng(G)$ such that $I$ is only contained in odd maximal subgroups, then $\type(X_I)=(1,1,0)$.
\end{proposition}

\begin{proof}
We proceed by structural induction on the structure classes. If $X_I$ is terminal, then $I$ is an odd maximal subgroup, so $\type(X_I)=(1,1,0)$.

Now, assume that $X_I$ is not terminal. Let $X_J$ be an option of $X_I$ and let $M$ be a maximal subgroup containing $J$. 
Then $I \leq J \leq M$ and we may conclude that $M$ is odd.  Thus, $\type(X_J)=(1,1,0)$ by induction.  
This implies that $X_I$ only has options of type $(1,1,0)$, and hence $\type(X_I)=(1,\mex\{0\},\mex\{1\})=(1,1,0)$.
\end{proof}

The next result, originally done with a different proof in~\cite[Proposition~3.22]{ErnstSieben}, is an immediate consequence of Proposition~\ref{prop:TypeOnlyInOddMaximals} since $\dng(G)$ is determined by the type of $X_{\Phi(G)}$. 

\begin{corollary}\label{cor:finiteOdd}
If $G$ is a nontrivial odd finite group, then $\dng(G)=*1$.
\end{corollary}


\end{section}


\begin{section}{Cyclic groups}


Odd cyclic cyclic groups were done in the previous section, and even cyclic groups are characterized by the following result.

\begin{proposition}\cite[Corollary 6.7]{ErnstSieben}
\label{prop:CyclicGroups}
If $G$ is an even cyclic group, then
\[
\dng(G) =  
   \begin{dcases} 
      *1, & |G|=2 \\
      *3, & 2\neq |G| \equiv_4 2 \\
      *0, & |G| \equiv_4 0.
   \end{dcases}
\]
\end{proposition}

The next two results are likely well-known to finite group theorists, but we provide their proofs for completeness. 

\begin{proposition}\label{prop:DirectProductMaximal} 
Let $G=H \times K$ for finite groups $H$ and $K$. If $M$ is a maximal subgroup of $H$, then $M \times K$ is a maximal subgroup of $G$.
\end{proposition}
\begin{proof}
Suppose that $L$ is a subgroup of $G$ such that $M \times K < L \leq H \times K=G$.  Let $(x,y) \in L \setminus (M \times K)$.  
Because $\{e\} \times K \le M \times K$, we can conclude that $(x,e) \notin M \times \{e\}$.  Then $\langle M,x  \rangle = H$ by the maximality of $M$.  We have $(x,e) \in L$ and $M \times \{e\} \leq L$ since $\{e\} \times K \leq L$, so $H \times \{e\}$ is a subgroup of $L$.  Then $H \times \{e\}$ and $\{e\} \times K$ are subgroups of $L$, so $G = H \times K \leq L$.  Therefore, $G=L$ and $M \times K$ is maximal in $G$.      
\end{proof}

\begin{proposition}\label{prop:4DirectProduct}
Let $G= P \times H$, where $|P|=2^n$ and $H$ is odd. Then every maximal subgroup of $G$ is even if and only if $n\ge 2$.
\end{proposition}

\begin{proof}
Let $M$ be a maximal subgroup of $P$.  First, suppose that $n \geq 2$.  Because $|P| \geq 4$, Cauchy's Theorem implies $|M| \geq 2$.  
Let $K$ be an odd subgroup of $G$.  Then $K\le \{e\}\times H < M \times H$, so $K$ is not maximal.  Therefore, every maximal subgroup is even.

Now, suppose that every maximal subgroup is even.  Then $M \times H$ is maximal by Proposition~\ref{prop:DirectProductMaximal}, and so it must have even order.  Then $|M| \geq 2$, so $|P| \geq 4$ and $n \geq 2$. 
\end{proof}

The following corollary will help tie the results for nim-numbers of cyclic groups to the results for nim-numbers of non-cyclic groups in Section~\ref{section:MainResults}.

\begin{corollary}\label{cor:CyclicEvenMaximals}
Let $G$ be a finite cyclic group. Then every maximal subgroup of $G$ is even if and only if $4$ divides $|G|$.
\end{corollary}
\begin{proof}
We may write $G$ as $G=P \times H$, where $P$ is a $2$-group and $H$ has odd order.  The result follows by Proposition~\ref{prop:4DirectProduct}.  
\end{proof}

\end{section}


\begin{section}{Non-cyclic Groups of even order}


Recall that a subset $\mathcal{C}$ of the power set of a group $G$ is a \emph{covering} of $G$ if $\bigcup \mathcal{C}=G$; in this case, we also say that $\mathcal{C}$ \emph{covers} $G$. 
Every proper subgroup of a finite group is contained in a maximal subgroup, but the set of maximal subgroups does not always cover the group.

We now consider even non-cyclic groups.  We will need Corollary~\ref{cor:InEvenMaximal} to do this, which follows immediately from Lemma~\ref{lem:CoverIffNoncyclic} and Lagrange's Theorem.

\begin{lemma}\label{lem:CoverIffNoncyclic} 
Let $G$ be a finite group.  Then the set of maximal subgroups of $G$ covers $G$ if and only if $G$ is non-cyclic. 
\end{lemma}
\begin{proof}
Suppose $G$ is non-cyclic and $x \in G$. Since $G$ is non-cyclic, $\langle x \rangle \not=G$.  Hence, $x \in \langle x \rangle \leq M$ for some maximal subgroup $M$ because there are only finitely many subgroups.  

Now, suppose $G$ is cyclic and let $x$ be a generator for $G$.  Then $x$ cannot be contained in any maximal subgroup $M$, lest $G=\langle x \rangle \leq M \neq G$.
\end{proof}

\begin{corollary}\label{cor:InEvenMaximal}
If $G$ is a finite non-cyclic group and $t \in G$ has even order, then $t$ is contained in an even maximal subgroup. 
\end{corollary}


We recall one of Barnes' main results in~\cite{Barnes}.

\begin{proposition}\cite[Theorem~1]{Barnes}
\label{prop:BarnesCriterion}
The first player wins $\dng(G)$ if and only if there is an element $g \in G$ of odd order such that $\langle g, t \rangle = G$ for every involution $t \in G$.
\end{proposition}

Note that Condition~(\ref{item:BarnesCondition}) in the next proposition is the denial of the condition used in Proposition~\ref{prop:BarnesCriterion}.  
This new proposition reformulates Barnes' condition about elements in terms of maximal subgroups. 

\begin{proposition}\label{prop:DNGMaximalCover}
If $G$ is an even non-cyclic group, then the following are equivalent.
\begin{enumerate}
\item\label{item:BarnesCondition}  There is no element $g \in G$ of odd order such that $\langle g, t \rangle = G$ for every involution $t\in G$.
\item  The set of even maximal subgroups covers $G$.
\end{enumerate}
\end{proposition}

\begin{proof}
Let $\mathcal{E}$ be the set of even maximal subgroups.  Suppose $\cup \mathcal{E} = G$, and let $g \in G$ have odd order.  Then there is some $L \in \mathcal{E}$ such that $g \in L$.  By Cauchy's Theorem, there is a $t \in L$ of order $2$.  Then $\langle g, t \rangle \leq L < G$.  So, there can be no element $g \in G$ of odd order such that $\langle g, t\rangle = G$ for all involutions $t \in G$.

Suppose $\cup \mathcal{E} < G$. Then there is an element $g \in G\setminus\left(\cup \mathcal{E}\right)$; this $g$ must have odd order by Corollary~\ref{cor:InEvenMaximal}.  Now, let $t \in G$ have order $2$. If $\langle g, t \rangle$ is not equal to $G$, then $\langle g ,t \rangle$ is contained in a maximal subgroup $M$.  Since $t$ has even order, it follows that $M$ must also be even, which contradicts the fact that $g$ is not an element of any even maximal subgroup.  Therefore, $\langle g,t \rangle$ is not contained in any maximal subgroup, which implies that $\langle g,t \rangle = G$.  
\end{proof}

The next corollary follows from Propositions~\ref{prop:BarnesCriterion} and \ref{prop:DNGMaximalCover} since $\dng(G)=*0$ if and only if the second player has a winning strategy.

\begin{corollary}
\label{cor:0iffCovering}
If $G$ is an even non-cyclic group, then the following are equivalent.
\begin{enumerate}
\item The set of even maximal subgroups covers $G$.
\item Every element of $G$ of odd order is in a proper even subgroup of $G$.
\item $\dng(G)=*0$.
\end{enumerate}
\end{corollary}


The following proposition shows that we only need to consider groups $G$ where the minimum number of generators $d(G)$ is 2.  One can prove this by slightly modifying the proof of~\cite[Theorem~1]{Barnes}, although we provide a different proof here.

\begin{proposition}\label{prop:bigdG}
If $G$ is an even group satisfying $\spr(G) \geq 3$, then $\dng(G)=*0$.
\end{proposition}

\begin{proof}
Suppose that the set of even maximal subgroups fails to cover $G$, and let $t \in G$ have order $2$. 
Let $\mathcal{N}$ be the set of maximal subgroups containing $t$. This is a nonempty set of even subgroups.
Because the entire set of even maximal subgroups fails to cover $G$, the set $\mathcal{N}$ must also fail to cover $G$.  
But then $\langle t,x \rangle = G$ for all $x \in G \setminus \bigcup \mathcal{N}$, so $\spr(G) \leq 2$.  
The result now follows from Corollary~\ref{cor:0iffCovering} since the set of even maximal subgroups covers $G$ when $\spr(G) \geq 3$.
\end{proof}

\begin{exam}
One can easily verify using GAP~\cite{GAP} that $\spr(G)=3$ if $G=\SYM(4)\times \SYM(4)\times \SYM(4)$, where $\SYM(k)$ is the symmetric group on $k$ letters.  Hence $\dng(G)=*0$.
\end{exam}


\end{section}


\begin{section}{Classification of avoidance games}\label{section:MainResults}


We seek a way to determine the nim-number of $\dng(G)$ for a finite group $G$ from covering properties of the maximal subgroups of $G$. 
We first recall the following result.

\begin{proposition}\cite[Proposition 3.20]{ErnstSieben}
\label{prop:spectrum}
The type of a structure class of $\dng(G)$ is in
\[
\{(0, 0, 1), (1, 0, 1), (1, 1, 0), (1, 3, 2) \}.
\]
\end{proposition}

The following allows us to complete the determination of nim-numbers for all avoidance games.

\begin{proposition}
\label{prop:NotCycNotCov}
Let $G$ be an even non-cyclic group. Then $\dng(G)=*3$ if and only if the set of even maximal subgroups does not cover $G$.
\end{proposition}

\begin{proof}
The contrapositive of the forward direction follows immediately from Corollary~\ref{cor:0iffCovering}. Now, assume that the set of even maximal subgroups does not cover $G$.
Since $G$ is non-cyclic, the set of maximal subgroups covers $G$, so there must be an odd maximal subgroup and hence $X_{\Phi(G)}$ is odd. We will prove that $\type(X_{\Phi(G)})=(1,3,2)$ by showing that both $(0, 0, 1)$ and $(1, 0, 1)$ are contained in 
the set \[T:=\{\type(X_I) \mid X_I \in \opt(X_{\Phi(G)})\}.\]    
An easy calculation shows that this is sufficient to determine the type of $X_{\Phi(G)}$, regardless of whether there are other elements in $T$ because $T$ is 
necessarily contained in the four-element set given in Proposition~\ref{prop:spectrum}.  Since $\emptyset \in X_{\Phi(G)}$, it will suffice to show 
that $\emptyset$ has options of type $(0,0,1)$ and $(1,0,1)$.

By Cauchy's Theorem, there is an involution $t \in G$. 
Because the set of even maximal subgroups does not cover $G$, there is an element $g \in G$ that is not contained in any even maximal subgroup.  
Because $G$ is non-cyclic, both $\{g\}$ and $\{t\}$ are options of the initial position $\emptyset$.  Let $\{g\}\in X_I$ and $\{t\}\in X_J$.
Because $g$ is only contained in odd maximal subgroups and $g\in I$, 
$\type(X_{I})=(1,1,0)$ by Proposition~\ref{prop:TypeOnlyInOddMaximals}. We also know that $X_{J}$ is even, so $\type(X_{J})=(0,0,1)$.   
\end{proof}

The following theorem is a complete categorization of the possible values of $\dng(G)$. 

\begin{theorem}\label{thm:DNGClassification}
Let $G$ be a nontrivial finite group.
\begin{enumerate}
\item  If $|G|=2$ or $G$ is odd, then $\dng(G)=*1$.
\item\label{thm:DNGClassification-2}  If $G\cong \mathbb{Z}_{4n}$ or the set of even maximal subgroups covers $G$, then $\dng(G)=*0$. 
\item Otherwise, $\dng(G)=*3$.
\end{enumerate}
\end{theorem}

\begin{proof}
The first case is handled by Proposition~\ref{prop:CyclicGroups} and Corollary~\ref{cor:finiteOdd}, respectively. 
For the second case, note that both conditions imply that $G$ has even order. Then the result follows from Proposition~\ref{prop:CyclicGroups} and Corollary~\ref{cor:0iffCovering}, respectively.

For the final case, assume that the conditions of Items~(1) and (2) are not met. If $G$ is cyclic, then $G\cong \mathbb{Z}_{4n+2}$ for some $n \geq 1$. Thus, $\dng(G)=*3$ by Proposition~\ref{prop:CyclicGroups}.  
On the other hand, if $G$ is not cyclic, then Proposition~\ref{prop:NotCycNotCov} implies that $\dng(G)=*3$, as well.
\end{proof}

Note that the set of maximal subgroups of a cyclic group never covers the group by Lemma~\ref{lem:CoverIffNoncyclic}. 
Therefore, exactly one of the hypotheses of Theorem~\ref{thm:DNGClassification}(\ref{thm:DNGClassification-2}) will hold.

\begin{corollary}\label{cor:DNGCriteria}
Let $G$ be a nontrivial finite group.
\begin{enumerate}
\item If all maximal subgroups of $G$ are odd, then $\dng(G)=*1$.
\item If all maximal subgroups of $G$ are even, then $\dng(G)=*0$.
\item Assume $G$ has both even and odd maximal subgroups.
  \begin{enumerate}
  \item If the set of even maximal subgroups covers $G$, then $\dng(G)=*0$.
  \item If the set of even maximal subgroups does not cover $G$, then $\dng(G)=*3$.
  \end{enumerate}
\end{enumerate}
\end{corollary}

\begin{proof}
Let $G$ be a finite group.  If all maximal subgroups of $G$ are odd, then it must be that $|G|$ is $2$ or odd, so $\dng(G)=*1$ by Theorem~\ref{thm:DNGClassification}.  If all maximal subgroups of $G$ are even and $G$ is cyclic, then $|G|$ is a multiple of $4$ by Corollary~\ref{cor:CyclicEvenMaximals} and $\dng(G)=*0$.  
If all maximal subgroups of $G$ are even and $G$ is not cyclic, then the set of maximal subgroups of $G$ covers $G$ and $\dng(G)=*0$ by Theorem~\ref{thm:DNGClassification}.  So, we may suppose that $G$ has both even and odd maximal subgroups.  

If the set of even maximal subgroups covers $G$, then $G$ cannot be cyclic by Lemma~\ref{lem:CoverIffNoncyclic}, and hence $\dng(G)=*0$ by Corollary~\ref{cor:0iffCovering}.  If the set of even maximal subgroups does not cover $G$, then $\dng(G)=*3$ by Theorem~\ref{thm:DNGClassification}. 
\end{proof}

\begin{exam}
All four cases of Corollary~\ref{cor:DNGCriteria} can occur.  All maximal subgroups of $\ZZ_3$ are odd, so $\dng(\ZZ_3)=*1$.  All maximal subgroups of $\ZZ_4$ are even, so $\dng(\ZZ_4)=*0$.  The group $G:=\ZZ_2 \times \ZZ_3 \times \ZZ_3$ has maximal subgroups of both parities with the set of even maximal subgroups covering $G$, so $\dng(G)=*0$.  The group $\SYM(3)$ has maximal subgroups of both parities with the set of even maximal subgroups failing to cover $\SYM(3)$, so $\dng(\SYM(3))=*3$.
\end{exam}

Corollary~\ref{cor:DNGCriteria} relates to an amusing result.   

\begin{proposition}
If the parity of all maximal subgroups of $G$ is the same, then the outcome of $\dng(G)$ does not depend on the strategy of the players.
\end{proposition}

\begin{proof}
If every maximal subgroup is even, the players will inevitably end up choosing from a single even maximal subgroup. Thus, regardless of strategy, the second player will win. A similar argument holds if every maximal subgroup is odd. 
\end{proof}

The parity of the maximal subgroups are all the same for $p$-groups. 
It has been conjectured that almost every group is a $p$-group; if so, there is almost always no strategy that changes the outcome of the game.

\begin{corollary}\label{cor:evenFrattini}
If $G$ is a finite group with an even Frattini subgroup, then $\dng(G)=*0$.
\end{corollary}
\begin{proof}
Since $\Phi(G)$ is even, every maximal subgroup must be even.  
\end{proof}

\begin{exam}
The Frattini subgroups of the general linear group $GL(2,3)$ and the special linear group $SL(2,3)$ are isomorphic to $\mathbb{Z}_2$. Hence the avoidance game on these groups has nim-number $0$.
\end{exam}

The following corollary is a restatement of the results above, but written as an ordered checklist of conditions, starting with what is easiest to verify.  Note that the hypothesis of Item~(\ref{item:PhiEven}) implies the hypothesis of Item~(\ref{item:AllEvenMaximals}), which in turn implies the hypothesis of Item~(\ref{item:EvenCovering}) if $G$ is non-cyclic, so there is some redundancy.  Items~(\ref{item:PhiEven}) and (\ref{item:AllEvenMaximals}) are listed separately because they are easier to check than Item~(\ref{item:EvenCovering}).  

\begin{corollary}
\label{cor:checklist}
Let $G$ be a nontrivial finite group. 
\begin{enumerate}
\item If $|G|=2$, then $\dng(G)=*1$.
\item If $G$ is odd, then $\dng(G)=*1$.
\item\label{item:PhiEven} If $\Phi(G)$ is even, then $\dng(G)=*0$.
\item\label{item:AllEvenMaximals} If every maximal subgroup of $G$ is even,  then $\dng(G)=*0$.
\item\label{item:EvenCovering} If the set of even maximal subgroups covers $G$, then $\dng(G)=*0$.
\item Otherwise, $\dng(G)=*3$.
\end{enumerate}
\end{corollary}

\begin{proof}
Items (1)--(6) follow immediately from Proposition~\ref{prop:CyclicGroups}, Corollary~\ref{cor:finiteOdd}, Corollary~\ref{cor:evenFrattini}, Corollary~\ref{cor:DNGCriteria}(2), Corollary~\ref{cor:DNGCriteria}(3)(a), and Corollary~\ref{cor:DNGCriteria}(3)(b), respectively.
\end{proof}

\end{section}


\begin{section}{Applications}\label{section:Applications}


\subsection{Nilpotent Groups}


Recall that finite nilpotent groups are the finite groups that are isomorphic to the direct product of their Sylow subgroups. It immediately follows that all abelian groups are nilpotent.  The next proposition generalizes the result for abelian groups found in~\cite{ErnstSieben}. 

\begin{proposition}\label{prop:nilpotent}
If $G$ is a nontrivial finite nilpotent group, then
\[
\dng(G)=
\begin{dcases}
*1, & |G|=2 \text{ or } G \text{ is odd}\\
*3, & G\cong \mathbb{Z}_2 \times \mathbb{Z}_{2k+1}, k\ge 1 \\
*0, & \text{otherwise}. 
\end{dcases}
\]
\end{proposition}

\begin{proof}
If $|G|=2$ or $G$ is odd, then $\dng(G)=*1$ by Theorem~\ref{thm:DNGClassification}. If $G \cong \ZZ_2 \times \ZZ_{2k+1}$ for some $k$, then it follows from Proposition~\ref{prop:CyclicGroups} that $\dng(G)=*3$.  

Lastly, we assume we are not in the first two cases;  then either $4$ divides $|G|$ or $G$ is non-cyclic of even order. We know $G \cong \prod_p Q_p$, where $Q_p$ is the Sylow $p$-subgroup of $G$ for the prime $p$.  
If $4$ divides $|G|$, then every maximal subgroup is even by Proposition~\ref{prop:4DirectProduct}, and hence $\dng(G)=*0$ by Corollary~\ref{cor:DNGCriteria}. 
So, suppose that $|Q_2|=2$ and $G$ is non-cyclic; then $\prod_{p\not=2} Q_p$ is not cyclic.  Let $g \in G$ have odd order.  Then $\langle g \rangle$ is a proper subgroup 
of $\prod Q_p$, so $Q_2 \times \langle g \rangle$ is a proper subgroup of $G$.  Thus, every element of odd order is contained in an even subgroup, so $\dng(G)=*0$ by Corollary~\ref{cor:0iffCovering}.
\end{proof}

\subsection{Generalized Dihedral Groups}


A group $G$ is said to be a generalized dihedral group if $G\cong A \rtimes \mathbb{Z}_2$ for some finite abelian $A$, where the action of the semidirect product is inversion.  In this case, we write $G=\dih(A)$ and we identify $A$ with the corresponding subgroup of $G$.  Note that $A$ has index $2$ in $G$, so $A$ is maximal in $G$. The proof of the following is a trivial exercise.

\begin{proposition}\label{prop:dihElements}
Every element of $\dih(A)$ that is not in $A$ has order $2$. 
\end{proposition}

\begin{proposition}\label{prop:dihMaximals}
Every maximal subgroup of $\dih(A)$ is even except possibly $A$. 
\end{proposition}
\begin{proof}
Let $M$ be a maximal subgroup of $\dih(A)$ that is not equal to $A$.  Then $M$ contains an element of $\dih(A)$ that is not contained in $A$.  By Proposition~\ref{prop:dihElements}, this element has order $2$, so $M$ is even. 
\end{proof}

\begin{proposition}\label{prop:gendih}
The avoidance games on the generalized dihedral groups satisfy: 
\[
\dng(\dih(A)) =\begin{dcases}
*3, & A \text{ is odd and cyclic} \\
*0, & \text{otherwise}.
\end{dcases}
\]
\end{proposition}
\begin{proof}
If $A$ is even, then all maximal subgroups of $\dih(A)$ are even by Proposition~\ref{prop:dihMaximals}.  Then $\dng(\dih(A))=*0$ by Corollary~\ref{cor:DNGCriteria}, so we may assume that $A$ is odd.

If $A=\langle a \rangle$ for some $a \in A$, then $\langle L, a \rangle =G$ for all even maximal subgroups $L$. Then  $a$ is not in the union of the even maximal subgroups,  and we conclude that $\dng(\dih(A))=*3$ by Theorem~\ref{thm:DNGClassification}.

If $A$ is non-cyclic,  let $g \in \dih(A)$ have odd order; then $g \in A$ by Proposition~\ref{prop:dihElements}.  Let $t \in \dih(A)$ be any element of order $2$.  Then $\langle g,t \rangle = \dih(\langle g \rangle) < \dih(A)$ since $A$ is not cyclic, so $\dng(G)=*0$ by Corollary~\ref{cor:0iffCovering}.
\end{proof}

Note that $\dng(\dih(A))$ can only be $*3$ if $\dih(A)$ is a dihedral group.  An alternative proof for Proposition~\ref{prop:gendih} is to determine the simplified structure diagrams shown in Figure~\ref{fig:gendih} for 
$\dng(\dih(A))$ in terms of the minimum number of generators $d(A)$ of $A$ and the parities of $A$ and $\Phi(A)$.

\begin{figure}
\begin{tabular}{ccccccccc}
\multicolumn{2}{c}{$A$ is even} & &  \multicolumn{3}{c}{$A$ is odd} \\
\includegraphics[scale=0.8]{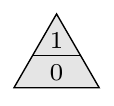} & \includegraphics[scale=0.8]{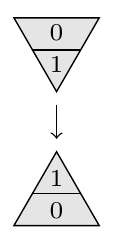}
&$\qquad$ &
\includegraphics[scale=0.8]{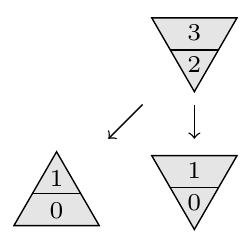} & \includegraphics[scale=0.8]{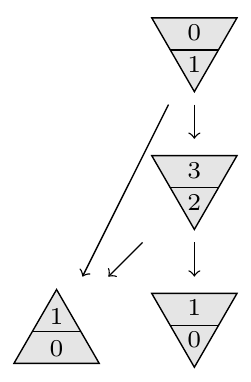} & \includegraphics[scale=0.8]{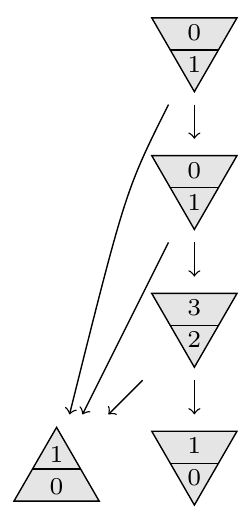}\tabularnewline
$\Phi(A)$ even & $\Phi(A)$ odd 
& & 
$\spr(A)=1$ &  $\spr(A)=2$ & $\spr(A)\ge3$\tabularnewline 
\end{tabular}

\protect\caption{\label{fig:gendih}
Simplified structure diagrams for $\protect\dng(\protect\dih(A))$. 
}
\end{figure}


\subsection{Generalized Quaternion Groups}


Recall that a group $G$ is a generalized quaternion group (or dicyclic group) if $G \cong \langle x,y \mid x^{2n}=y^4=1, x^n=y^2,x^y=x^{-1}\rangle$ for some $n \geq 2$. The quaternion group of order $8$ is the $n=2$ case.  

\begin{proposition}
If $G$ is a generalized quaternion group, then $\dng(G)=*0$.
\end{proposition}

\begin{proof}
Let $g \in G$ have odd order.  Using notation from the presentation above, one can easily verify that $G=X \cup Xy$, where $X=\langle x\rangle$.  It is easy to check that every element of $Xy$ has order $4$ and so $g \in X$.  Since $X$ is even, $\dng(G)=*0$ by Corollary~\ref{cor:0iffCovering}.
\end{proof}


\subsection{Groups with Real Elements}


Recall~\cite{Rose} that an element $g$ of a group $G$ is \emph{real} if there is a $t \in G$ such that $g^t=g^{-1}$.  Note that $t$ induces an automorphism of order $2$ on $\langle g \rangle$, so $G$ must be even if it has a real element.  

\begin{proposition}\label{prop:RealElements}
If $G$ is a finite group such that $|G| > 2$ and $g \in G$ is a real element of odd order, then $g$ is contained in a proper even subgroup or $G \cong \dih(\langle g \rangle)$. 
\end{proposition}
\begin{proof}
Let $K=\langle g \rangle$.  Because $g$ is real, there is a $t \in G$ such that $g^t=g^{-1}$.  Then conjugation by $t$ in $G$ induces an automorphism of $K$ of order $2$, so the order of $t$ is even.  Since $t$ normalizes $K$, we may define a subgroup $L= K\langle t \rangle$.  If $L < G$, then $g \in L < G$ and we are done since $L$ must be even.  

So assume that $G=L$ with $K$ normal in $G$.  Let $u$ be an involution in $G$, and let $M=K\langle u \rangle$.  If $M < G$, then $g \in M < G$ and we are done since $M$ is even.  So assume that $G=K\langle u \rangle=\langle g,u \rangle$ with $|G|=2|\langle g \rangle|$.  Since $g$ is real and $G=K \cup Ku$, $u$ must invert $g$ and we have $G \cong \dih(\langle g \rangle)$. 
\end{proof}


The next corollary immediately follows from Proposition~\ref{prop:CyclicGroups},~Corollary~\ref{cor:0iffCovering}  and Propositions~\ref{prop:gendih} and~\ref{prop:RealElements}.


\begin{corollary}\label{cor:RealGroups}
If $G$ is a finite group such that every element of $G$ of odd order is real, then 
\[
\dng(G) =\begin{dcases}
*1, & |G|=2\\
*3, & \text{$G\cong\dih(A)$ for some odd cyclic $A$}\\
*0, & \text{otherwise}.
\end{dcases}
\]
\end{corollary}

\begin{exam}
If $G$ is any of the groups listed below, then $\dng(G)=*0$ by Corollary~\ref{cor:RealGroups} since every element of $G$ is real by~\cite{TiepZalesski}.
\begin{enumerate}
\item $\Omega_{2n+1}(q)$ with $q \equiv_4 1$ and $n \geq 3$ 
\item $\Omega_{9}^+(q)$ with $q \equiv_4 3$
\item P$\Omega_{4n}^+(q)$ with $q \not\equiv_4 3$ and $n \geq 3$
\item $\Omega_{4n}^+(q)$ with $q \not\equiv_4 3$ and $n \geq 3$
\item $\prescript{3}{}{\operatorname{D}}_4(q)$
\item any quotient of $\operatorname{Spin}_{4n}^-(q)'$ with $n \geq 2$
\item any quotient of $\operatorname{Sp}_{2n}(q)'$ with $q \not\equiv_4 3$ and $n \geq 1$
\end{enumerate}

\end{exam}

Proposition~\ref{prop:RealElements} implies that $\dng(\SYM(n))=*0$ for all $n \geq 4$ since every element in $\SYM(n)$ is conjugate to other elements with the same cycle structure, including the element's inverse.  This is generalized in the following corollary for Coxeter groups.  Recall that the Coxeter groups of types $A_n$ and $I_2(m)$ are isomorphic to $\SYM(n+1)$ and the dihedral group of order $2m$, respectively~\cite{Humphreys1990}.

\begin{corollary}\label{cor:Coxeter}
If $G$ is a finite irreducible Coxeter group, then  
\[
\dng(G) =\begin{dcases}
*1, & \text{$G$ is of type $A_1$}\\
*3, & \text{$G$ is of type $I_2(m)$ for some odd $m$}\\
*0, & \text{otherwise}.
\end{dcases}
\]
\end{corollary}
\begin{proof}
If $G$ is of type $A_1$, then $G$ is isomorphic to $\ZZ_2$ and $\dng(G)=*1$ by Proposition~\ref{thm:DNGClassification}.  If $G$ is of type $I_2(m)$ for some odd $m$ (which includes type $A_2$), then $G$ is dihedral of order $2m$, and hence $\dng(G)=*3$ by Proposition~\ref{prop:gendih}.  Every element of every other Coxeter group is real by~\cite[Corollary~3.2.14]{MR1778802}, so the result follows from Corollary~\ref{cor:RealGroups}. 
\end{proof}


\subsection{Alternating Groups}


Let $\ALT(n)$ denote the alternating group on $n$ letters.  The last two authors proved in~\cite{ErnstSieben} that $\dng(\ALT(3))=*3=\dng(\ALT(4))$.  Barnes proved in \cite{Barnes} that $\dng(\ALT(n))=*0$ if $n$ is greater than $5$ and not a prime that is congruent to $3$ modulo $4$.  The full characterization of $\dng(\ALT(n))$ can be found in \cite{BeneshErnstSiebenSymAlt}, which uses the O'Nan--Scott Theorem~\cite{AschbacherScott} to show 
\[
\dng(\ALT(n)) =\begin{dcases}
*3, & n\in\{3,4\} \text{ or } n \text{ is in a certain family of primes} \\
*0, & \text{otherwise}.
\end{dcases}
\]

\noindent
The smallest number in this family of primes is $19$.


\subsection{Sporadic Groups}


The sporadic groups are the finite simple groups that do not belong to any infinite family of simple groups.
\begin{proposition}
Let $M_{23}$ denote the Mathieu group that permutes $23$ objects and $B$ denote the Baby Monster. If $G$ is a sporadic group, then
\[
\dng(G) =\begin{dcases}
*3, & G \cong M_{23} \text{ or } G \cong B\\
*0, & \text{otherwise.} 
\end{dcases}
\]
\end{proposition} 
\begin{proof}
Suppose $G \cong M_{23}$.  Then $23$ divides $|G|$, so $G$ has an element of order $23$.  The Atlas of Finite Groups~\cite{Atlas} states that no even maximal subgroup of $G$ has order divisible by $23$, so we conclude that elements of order $23$ are not in any even maximal subgroup and the set of even maximal subgroups does not cover $G$.  Then $\dng(G)=*3$ by Corollary~\ref{cor:DNGCriteria}. If $G \cong B$, then similarly elements of order $47$ are not contained in any even maximal subgroup, since no even maximal subgroup is divisible by $47$~\cite{wilson1999maximal}. So $\dng(G)=*3$ in this case, too. 

Now suppose that $G$ is isomorphic to the Thompson group $Th$.  There is only one class of maximal subgroups of $G$ that is odd~\cite{linton1989maximal}, and these odd maximal subgroups are isomorphic to $31.15$.  Thus, we only need to consider elements of order dividing $31 \cdot 15$, since all other elements are contained in some even maximal subgroup.  By the character table in~\cite{Atlas}, it suffices to consider elements of order $3$, $5$, $15$, and $31$.  Again by the character table, all elements of orders $3$, $5$, and $15$ are contained in centralizers of even order.  Therefore, all elements of $G$ are contained in an even maximal subgroup except possibly for elements of order $31$.

So let $M$ be a maximal subgroup of $G$ isomorphic to $31.15$, and let $g \in M$ have order $31$.  By the character table, there are exactly two conjugacy classes $C_1$ and $C_2$ of elements of order $31$, and any element of order $31$  is contained in one and its inverse is contained in the other.    By~\cite{Atlas}, there is a maximal subgroup $H$ of $G$ that is isomorphic to $2^5.L_5(2)$, so $31$ divides $|H|$ and $H$ has an element $h$ of order $31$.  Then $g$ is conjugate to either $h$ or $h^{-1}$; without loss of generality, assume that there is an $x \in G$ such that $h^x=g$.  Then $g = h^x \in H^x$, so $g$ is contained in an even maximal subgroup.  Therefore, the maximal subgroups of $G$ cover $G$, so $\dng(G)=*0$ by Corollary~\ref{cor:DNGCriteria}.

The remaining cases contain only even maximal subgroups, and therefore have nim-number $0$ by Corollary~\ref{cor:DNGCriteria}:
\begin{enumerate} 
\item The group $J_4$ has only even maximal subgroups by~\cite{kleidman1988maximal}.
\item The group $Fi_{23}$ has only even maximal subgroups by~\cite{kleidman1989maximal}.
\item The group $Fi_{24}'$ has only even maximal subgroups by~\cite{linton1991maximal}.
\item The group $M$ has only even maximal subgroups by~\cite{norton1998anatomy},~\cite{norton2002anatomy},~\cite{norton2013correction},~\cite{wilson1999maximal}, and~\cite{wilson2006new}.
\end{enumerate} 

Every maximal subgroup is even for every sporadic group not already mentioned~\cite{Atlas}, so they also have nim-number $0$. 
\end{proof}

\subsection{Rubik's Cube Groups}


We can use our classification results together with a computer algebra system to handle some fairly large groups. There are $8$ and $20$ conjugacy classes of maximal subgroups of the $2\times2\times 2$ and $3\times3\times 3$ Rubik's Cube groups, respectively. A simple GAP calculation~\cite{WEB2} shows that all these maximal subgroups are even. Hence, the avoidance games on these groups are all $*0$.

\end{section}


\begin{section}{Groups with Odd Order Frattini Subgroup}


If the Frattini subgroup of a group $G$ is even, then $\dng(G)=*0$ by Corollary \ref{cor:evenFrattini}. 
If the Frattini subgroup is odd, then factoring out by the Frattini subgroup does not change the nim-number of $\dng(G)$. 

\begin{proposition}\label{prop:FrattiniQuotients}
Let $G$ be a finite group, and let $N$ be a normal subgroup of $G$ such that $N$ is odd and $N \leq \Phi(G)$.  Then $\dng(G)=\dng(G/N)$.
\end{proposition}

\begin{proof}
The result follows trivially if $|G|=2$ and follows from Corollary~\ref{cor:finiteOdd} if $G$ is odd, so assume that $G$ is even of order greater than $2$.  If $G$ is cyclic, then $4$ divides $|G|$ if and only if $4$ divides $|G/N|$, since $N$ is odd; the result follows by Proposition~\ref{prop:CyclicGroups}.  So, suppose that $G$ is non-cyclic.

Let $x \in G$.  If $M$ is a maximal subgroup of $G$, then $N \leq \Phi(G) \leq M$ by the definition of $\Phi(G)$.  By the Correspondence Theorem~\cite[Theorem~3.7]{Isaacs1994}, the maximal subgroups of $G/N$ that contain $Nx$ are exactly the set of subgroups of the form $M/N$, where $M$ is a maximal subgroup of $G$ containing $x$.  Additionally, $|(MN)/N|=|M/N|=|M|/|N| \equiv_2 |M|$ since $N$ is odd.  
Therefore, the parities of the orders of maximal subgroups of $G/N$ that contain $Nx$ are exactly the same as the parities of orders of maximal subgroups of $G$ that contain $x$, so the set of even maximal subgroups of $G/N$ covers $G/N$ if and only if the set of even maximal subgroups of $G$ covers $G$.  The result follows from Theorem~\ref{thm:DNGClassification}. 
\end{proof}

It is well-known that the Frattini quotient of a nilpotent group is abelian. Since the avoidance games for abelian groups were classified in~\cite{ErnstSieben}, we could have used Proposition~\ref{prop:FrattiniQuotients} to prove Proposition~\ref{prop:nilpotent}.

\begin{figure}
\begin{tabular}{cccc}
\includegraphics{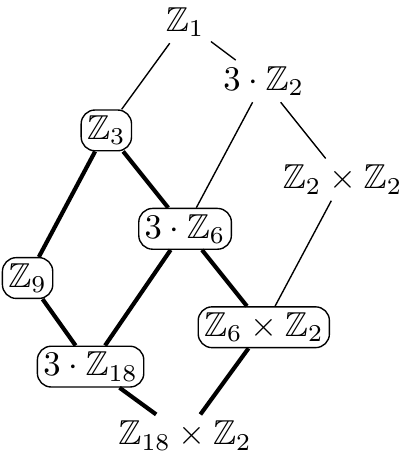} & \includegraphics{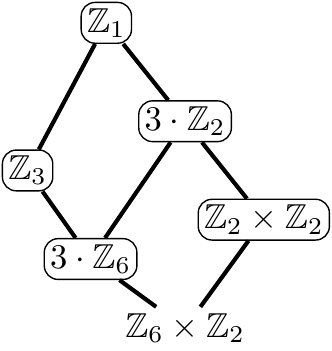} & \includegraphics[scale=.8]{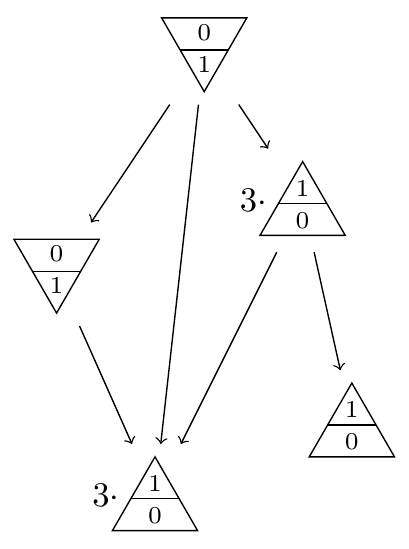} &  \includegraphics[scale=.8]{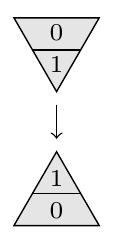}  \\
(a) & (b) & (c) & (d)
\end{tabular}
\caption{\label{fig:C18xC2}The subgroup lattices for $\mathbb{Z}_{18}\times\mathbb{Z}_2$ and its quotient group $\mathbb{Z}_{6}\times\mathbb{Z}_2$, together with their common structure diagram 
and simplified structure diagram for the avoidance game. The intersection subgroups are framed. Note that $3 \cdot \mathbb{Z}_{n}$ refers to $3$ distinct copies of $\ZZ_{n}$ at the same node in the diagram. }
\end{figure}

\begin{exam}
The Frattini subgroup of $G=\mathbb{Z}_{18}\times \mathbb{Z}_2$ is $\Phi(G)\cong\mathbb{Z}_3$. Hence, $\dng(G)=\dng(G/\Phi(G))=\dng(\mathbb{Z}_6\times \mathbb{Z}_2)=*0$, where the last equality follows from Proposition~\ref{prop:nilpotent}. Figure~\ref{fig:C18xC2} shows how the subgroup lattice
changes in the factoring process while the structure diagrams remain the same. 
Notice that there are three arrows in Figure~\ref{fig:C18xC2}(c) joining the structure classes corresponding to the intersection subgroups
$\mathbb{Z}_1$ and $3\cdot\mathbb{Z}_6$ in the structure diagram of $\dng(\mathbb{Z}_6\times \mathbb{Z}_2)$, even though there is no edge joining $\mathbb{Z}_1$ to any copy of $\mathbb{Z}_6$ in Figure ~\ref{fig:C18xC2}(b). A similar relationship holds between the structure classes corresponding to the intersection subgroups $\mathbb{Z}_3$ and $3\cdot\mathbb{Z}_{18}$ for $\dng(\mathbb{Z}_{18}\times \mathbb{Z}_2)$.
\end{exam}

Factoring out by the Frattini subgroup can reduce the size of the group to make it more manageable for computer calculation. 

\begin{exam}
The Frattini subgroup of the special linear group $SL(3,7)$ is isomorphic to $\mathbb{Z}_3$. The Frattini quotient is isomorphic to the projective special linear group $PSL(3,7)$. 
A GAP calculation shows that the Frattini quotient has $69008$ even and $32928$ odd maximal subgroups.
The set of even maximal subgroups does not cover $PSL(3,7)$, so $\dng(SL(3,7))=\dng(PSL(3,7))=*3$ by Corollary~\ref{cor:DNGCriteria}(3)(b). 
\end{exam}


\end{section}


\begin{section}{Further Questions}\label{sec:questions}

Below we outline a few open problems related to $\dng(G)$.
\begin{enumerate}
\item The nim-number of the avoidance game is determined by the structure diagram, which is determined by the structure digraph and the parity of the intersection subgroups.  
Is it possible to determine the structure diagram from the abstract subgroup lattice structure without using any information about the subgroups?
\item Can we characterize $\dng(G\times H)$, or even $\dng(G\rtimes H)$, in terms of $\dng(G)$ and $\dng(H)$?
\item Let $N$ be a normal subgroup of $G$ such that $N\le \Phi(G)$. Are the structure digraphs of $\dng(G)$ and $\dng(G/N)$ isomorphic?
\item Can we characterize the nim-numbers of the achievement games from~\cite{ErnstSieben} in terms of covering conditions by maximal subgroups similar to Corollary~\ref{cor:DNGCriteria} for avoidance games?
\item Is it possible to determine nim-numbers for avoidance games played on algebraic structures having maximal sub-structures, such as quasigroups, semigroups, monoids, and loops?
\item The nim-number of a game position $P$ can be determined using the type of the structure class containing $P$. The type of the structure class requires a recursive computation using the minimal excludant.
Is it possible to avoid this computation and determine the nim-value of $P$ in terms of a simple condition using only the maximal subgroups? 
\item What are the nim-values of avoidance games played on the remaining classical and simple groups?   
\end{enumerate}

\end{section}

\bibliographystyle{amsplain}
\bibliography{game}

\end{document}